\documentclass[pdflatex,sn-mathphys-num]{sn-jnl}% Math and Physical Sciences Numbered Reference Style
\usepackage{amsmath}
\usepackage{amssymb}
\usepackage{enumerate}
\usepackage{algorithm,algorithmic}
\usepackage{caption}
\usepackage{array}
\usepackage{pgfplots, pgfplotstable, booktabs, colortbl, array}
\usepackage{multirow}

\DeclareMathOperator{\tr}{tr}
\usepackage{float}
\floatstyle{plaintop}
\restylefloat{table}

\usepackage{todonotes}

\newcommand{\ignore}[1]{}

\usepackage{amsfonts}
\usepackage{amssymb}
\newtheorem{remark}{Remark}
\newcommand{\norm}[1]{\|#1\|}
\newcommand{\abs}[1]{|#1|}

\newcommand{\uu}{{\bf u}}


\newtheorem{theorem}{Theorem}[section]
\newtheorem{lemma}[theorem]{Lemma}

\theoremstyle{definition}
\newtheorem{definition}[theorem]{Definition}

\theoremstyle{remark}

\theoremstyle{thmstyleone}%

\theoremstyle{thmstyletwo}%

\theoremstyle{thmstylethree}%

\begin{document}

\title[Shifted CholeskyQR for sparse matrices]{Shifted CholeskyQR for sparse matrices}

%%=============================================================%%
%% GivenName	-> \fnm{Joergen W.}
%% Particle	-> \spfx{van der} -> surname prefix
%% FamilyName	-> \sur{Ploeg}
%% Suffix	-> \sfx{IV}
%% \author*[1,2]{\fnm{Joergen W.} \spfx{van der} \sur{Ploeg} 
%%  \sfx{IV}}\email{iauthor@gmail.com}
%%=============================================================%%

\author*[1]{\fnm{Haoran} \sur{Guan}}\email{21037226R@connect.polyu.hk}

\author[2]{\fnm{Yuwei} \sur{Fan}}\email{fanyuwei2@huawei.com}

\affil*[1]{\orgdiv{Department of Applied Mathematics}, \orgname{The Hong Kong Polytechnic University}, \orgaddress{\street{Hung Hom}, \city{Kowloon}, \postcode{999077}, \state{Hong Kong SAR}, \country{China}}}

\affil[2]{\orgdiv{Theory Lab}, \orgname{Huawei Hong Kong Research Center}, \orgaddress{\street{Sha Tin}, \city{New Territories}, \postcode{999077}, \state{Hong Kong SAR}, \country{China}}}

%%==================================%%
%% Sample for unstructured abstract %%
%%==================================%%

\abstract{In this work, we focus on Shifted CholeskyQR (SCholeskyQR) for sparse matrices. We provide a new shifted item $s$ for Shifted CholeskyQR3 (SCholeskyQR3) based on the number of non-zero elements (nnze) and the element with the largest absolute value of the input sparse $X \in \mathbb{R}^{m\times n}$ with $m \ge n$. We do rounding error analysis of SCholeskyQR3 with such an $s$ and show that SCholeskyQR3 is accurate in this case. Therefore, an alternative choice of $s$ can be taken for SCholeskyQR3 with the comparison between our new $s$ and the $s$ shown in the previous work when the input $X$ is sparse, improving the applicability and residual of the algorithm for the ill-conditioned cases. Numerical experiments demonstrate the advantage of SCholeskyQR3 with our alternative choice of $s$ in both applicablity and accuracy over the case with the original $s$, together with the same level of efficiency. This work is also the first to build connections between sparsity and numerical algorithms with detailed rounding error analysis to the best of our knowledge.}


%%\pacs[JEL Classification]{D8, H51}

%%\pacs[MSC Classification]{35A01, 65L10, 65L12, 65L20, 65L70}

\maketitle

\section{Introduction}
The problem of matrix factorization is encountered in both academia and industry across various fields, such as data analysis and engineering. Numerous well-known algorithms exist for matrix factorization, including QR factorization, Cholesky factorization, and LU factorization. QR factorization is one of the most important methods of matrix factorization and is particularly useful in many applications, such as randomized singular value decomposition, Krylov subspace methods, the local optimal block preconditioned conjugate gradient method (LOBPCG), and block Householder QR algorithms  \cite{halko2011, 2020, 2018robust, 1989}. In recent years, several different algorithms for QR factorization have been developed, each with distinct advantages, including Householder QR, CGS(2), MGS(2), TSQR, and CholeskyQR. For further details, see \cite{ballard2011, 2011, Communication, MatrixC, Higham, Numerical} and their references.

\subsection{Summary of CholeskyQR-type algorithms}
Among all the algorithms for QR factorization, CholeskyQR, which is shown in Algorithm~\ref{alg:cholqr}, has attracted much attention in recent years. It is a communication-avoiding algorithm specifically designed for $X \in \mathbb{R}^{m\times n}$ with $m \ge n$ and $\mbox{rank(X)}=n$. CholeskyQR outperforms TSQR in terms of speed \cite{2014}. For the input matrix $X \in \mathbb{R}^{m\times n}$, the first step of CholeskyQR involves computing a gram matrix $G \in \mathbb{R}^{n\times n}$. Subsequently, a step of Cholesky factorization is performed to obtain an upper-triangular matrix $R \in \mathbb{R}^{n\times n}$, from which $Q \in \mathbb{R}^{m\times n}$ can then be computed.

\begin{algorithm}
\caption{$[Q,R]=\mbox{CholeskyQR}(X)$}
\label{alg:cholqr}
\begin{algorithmic}[1]
\REQUIRE $X \in \mathbb{R}^{m\times n}.$ 
\ENSURE \mbox{Orthogonal factor} $Q \in \mathbb{R}^{m\times n}$, \mbox{Upper triangular factor} $R \in \mathbb{R}^{n \times n}.$ 
\STATE $G  = X^{\top}X,$
\STATE $R = \mbox{Cholesky}(G),$
\STATE $Q  =  XR^{-1}.$
\end{algorithmic}
\end{algorithm}% 

However, a single CholeskyQR does not have good orthogonality in orthogonality when $\kappa_{2}(X)$ is large enough. Therefore, Algorithm~\ref{alg:cholqr} is seldom used directly. Alternatively, an improved algorithm called CholeskyQR2 \cite{2014, error} has been developed by applying CholeskyQR twice, as shown in Algorithm~\ref{alg:cholqr2}. CholeskyQR2 has good accuracy in both orthogonality and residual. However, with rounding errors in the matrix multiplication and matrix factorization, CholeskyQR2 is not applicable to many ill-conditioned matrices, as numerical breakdown may occur during Cholesky factorization. To address this issue, researchers have proposed a novel enhanced algorithm known as Shifted CholeskyQR (SCholeskyQR) with a shifted item $s$ to deal with ill-conditioned cases \cite{Shifted}, as detailed in Algorithm~\ref{alg:Shifted}. CholeskyQR2 is then applied after SCholeskyQR and Shifted CholeskyQR3 (SCholeskyQR3) is formed, which keeps accuracy in both orthogonality and residual, as shown in Algorithm~\ref{alg:Shifted3}. For SCholeskyQR3, we provide an improved SCholeskyQR3 with a better shifted item $s$ in \cite{New}, providing better applicability with the same level of accuracy and efficiency compared to the $s$ in \cite{Shifted}.

\begin{algorithm}
\caption{$[Q,R]=\mbox{CholeskyQR2}(X)$}
\label{alg:cholqr2}
\begin{algorithmic}[1]
\REQUIRE $X \in \mathbb{R}^{m\times n}.$ 
\ENSURE \mbox{Orthogonal factor} $Q \in \mathbb{R}^{m\times n}$, \mbox{Upper triangular factor} $R \in \mathbb{R}^{n \times n}.$ 
\STATE $[W,Y]=\mbox{CholeskyQR}(X),$
\STATE $[Q,Z]=\mbox{CholeskyQR}(W),$
\STATE $R=ZY.$
\end{algorithmic}
\end{algorithm}%

\begin{algorithm}
\caption{$[Q,R]=\mbox{SCholeskyQR}(X)$}
\label{alg:Shifted}
\begin{algorithmic}[1]
\REQUIRE $X \in \mathbb{R}^{m\times n}.$ 
\ENSURE \mbox{Orthogonal factor} $Q \in \mathbb{R}^{m\times n}$, \mbox{Upper triangular factor} $R \in \mathbb{R}^{n \times n}.$ 
\STATE $G  = X^{\top}X,$
\STATE take $s>0,$
\STATE $R = \mbox{Cholesky}(G+sI),$
\STATE $Q  =  XR^{-1}.$
\end{algorithmic}
\end{algorithm}%

\begin{algorithm}
\caption{$[Q,R]=\mbox{SCholeskyQR3}(X)$}
\label{alg:Shifted3}
\begin{algorithmic}[1]
\REQUIRE $X \in \mathbb{R}^{m\times n}.$ 
\ENSURE \mbox{Orthogonal factor} $Q \in \mathbb{R}^{m\times n}$, \mbox{Upper triangular factor} $R \in \mathbb{R}^{n \times n}.$ 
\STATE $[W,Y]=\mbox{SCholeskyQR}(X),$
\STATE $[Q,Z]=\mbox{CholeskyQR2}(W),$
\STATE $R=ZY.$
\end{algorithmic}
\end{algorithm}%

\subsection{New considerations}
In \cite{Shifted}, the requirement of $\kappa_{2}(X)$ for SCholeskyQR3 is derived from some steps of error estimations. The upper bound of $\kappa_{2}(X)$ is expressed as the polynomial with $m$ and $n$ for $X \in \mathbb{R}^{m\times n}$. Moreover, the shifted item $s$ for SCholeskyQR3 is taken based on the upper bounds of several certain steps of error analysis. In \cite{New}, we introduce a new $\norm{\cdot}_{g}$ and provide better error analysis compared to that in \cite{Shifted}. Therefore, an improved shifted item $s$ based on $\norm{\cdot}_{g}$ is provided for SCholeskyQR3, improving the applicability while maintaining accuracy and efficiency. However, we still aim to provide better rounding error analysis for SCholeskyQR3, which can bring a more optimal $s$ to enhance the properties of the algorithm.

In many real-world applications, especially in industry and various scientific fields, the input matrix $X \in \mathbb{R}^{m\times n}$ is often sparse, especially when $m$ and $n$ are large. Sparse matrices exhibit different properties compared to dense matrices and frequently arise in numerical PDEs and their applications in physics, chemistry and astronomy. Regarding CholeskyQR-type algorithms, we are exploring whether the sparsity of $X$ can exhibit different properties and performances compared to the dense cases. Since the sparsity of $X$ may lead to overestimation of error bounds in many steps of the algorithms, we try to provide theoretical analysis of SCholeskyQR in sparse cases and take a alternative shifted item $s$ for the sparse $X$ based on its number of non-zero elements (nnze).

\subsection{Our contributions in this work}
To the best of our knowledge, this work is the first to discuss the connection between CholeskyQR-type algorithms and sparse matrices. It is hard for us to do theoretical analysis for sparse matrices quantitatively with the tools in \cite{Shifted, New}. Therefore, we combine the sparsity and the key element of the sparse matrices with rounding error analysis, which is highly innovative compared to the existing works. We propose a new $s$ based on the nnze and the key element with the largest absolute value of the input $X$ for SCholeskyQR3, which differs significantly from the approaches in \cite{New, Shifted}. Therefore, an alternative choice of $s$ based on the comparison between our new $s$ and the original $s$ in \cite{New} can be taken. We prove that such a new $s$ can improve the applicability and residual of SCholeskyQR3 compared to the results in \cite{New}. Numerical experiments confirm the advantage of applicability and accuracy of SCholeskyQR3 with our new $s$ and show that our alternative choice of $s$ can guarantee the efficiency of the algorithm. Distinguished from many existing works about sparse matrices, the primary purpose of this work is not to accelerate the algorithm but to improve the applicability of the algorithm based on the properties of the sparse matrices, showing a different perspective on the research of sparsity. Our work is an extension and improvement of \cite{New} in the widely-used sparse cases and is the only work focusing on rounding error analysis of the sparse cases as far as we know.

Here, we give some notations of this work. $\norm{\cdot}_{F}$ and $\norm{\cdot}_{2}$ denote the Frobenius norm and the $2$-norm of the matrix. For $X \in \mathbb{R}^{m\times n}$ and is full-rank, $\norm{X}_{2}=\sigma_{1}(X)$ is the largest singular value and $\sigma_{min}(X)$ is the smallest singular value of $X$. $\kappa_{2}(X)=\frac{\norm{X}_{2}}{\sigma_{min}(X)}$ is the condition number of $X$. $\uu=2^{-53}$ is the unit roundoff. For the input matrix $X$, $\abs{X}$ is the matrix whose elements are all the absolute values of the elements of $X$. $fl(\cdot)$ denotes the computed value in floating-point arithmetic. 

\subsubsection{The improved SCholeskyQR3}
Before introducing the theoretical results in this work, we make a review of the improved SCholeskyQR3 (ISCholeskyQR3) with a new $s$ in our previous work \cite{New} compared to that in \cite{Shifted}. This is an $s$ for the general cases. We provide a definition of a new $\norm{\cdot}_{g}$ in \cite{New} first.

\begin{definition}[The definition of ${\norm{\cdot}_{g}}$]
\label{def:defg}
If $X=[X_{1},X_{2}, \cdots X_{n-1},X_{n}] \in \mathbb{R}^{m\times n}$, then
\begin{equation}
\norm{X}_{g}:=\max_{1 \le j \le n}\norm{X_{j}}_{2}. \nonumber
\end{equation}
where
\begin{equation}
\norm{X_{j}}_{2}=\sqrt{x_{1,j}^{2}+x_{2,j}^{2}+ \cdots +x_{m-1,j}^{2}+x_{m,j}^{2}}. \nonumber
\end{equation}
\end{definition}

With Definition~\ref{def:defg}, we show the theoretical results of ISCholeskyQR3 with a shifted item based on $\norm{X}_{g}$ for the input $X \in \mathbb{R}^{m\times n}$.

\begin{lemma}[The relationship between $\kappa_{2}(X)$ and $\kappa_{2}(W)$ for ISCholeskyQR]
\label{lemma 2.9}
For $X \in \mathbb{R}^{m\times n}$ and $[W,Y]=\mbox{ISCholeskyQR}(X)$, with $mn\uu \le \frac{1}{64}$, $n(n+1)\uu \le \frac{1}{64}$, $11(mn\uu+n(n+1)\uu)\norm{X}_{g}^{2} \le s \le \frac{1}{100}\norm{X}_{g}^{2}$ and $\kappa_{2}(X) \le \frac{1}{4.89pn^{2}\uu}$, we have
\begin{equation}
\kappa_{2}(W) \le 3.24\sqrt{1+t(\kappa_{2}(X))^{2}}. \nonumber
\end{equation}
Here, we have $t=\frac{s}{\norm{X}_{2}^{2}} \le \frac{1}{100}$. When $[Q,R]=\mbox{ISCholeskyQR3}(X)$, if we take $s=11(mn\uu+n(n+1)\uu)\norm{X}_{g}^{2}$ and $\kappa_{2}(X)$ is large enough, a sufficient condition for $\kappa_{2}(X)$ is
\begin{equation}
\kappa_{2}(X) \le \frac{1}{86p(mn\uu+(n+1)n\uu)} \le \frac{1}{4.89pn^{2}\uu}. \label{eq:sck}
\end{equation}
\end{lemma}

\begin{lemma}[Rounding error analysis of ISCholeskyQR3]
\label{lemma 2.10}
For $X \in \mathbb{R}^{m\times n}$ and $[Q,R]=\mbox{ISCholeskyQR3}(X)$, if we take $s=11(mn\uu+n(n+1)\uu)\norm{X}_{g}^{2}$ and \eqref{eq:sck} is satisfied, with $mn\uu \frac{1}{64}$ and $n(n+1)\uu \le \frac{1}{64}$, we have
\begin{align}
\norm{Q^{\top}Q-I}_{F} &\le 6(mn\uu+n(n+1)\uu), \nonumber \\
\norm{QR-X}_{F} &\le (6.57p+4.81)n^{2}\uu\norm{X}_{2}. \nonumber
\end{align}
Here, $p=\frac{\norm{X}_{g}}{\norm{X}_{2}}$, $\frac{1}{\sqrt{n}} \le p \le 1$. 
\end{lemma}

\subsubsection{SCholeskyQR3 for sparse matrices}
Different from ISCholeskyQR3, we propose a new shifted item $s$ for SCholeskyQR3 when the input $X \in \mathbb{R}^{m\times n}$ is sparse. When $X \in \mathbb{R}^{m\times n}$ is a sparse matrix, we give some general settings below.
\begin{align}
mn\uu &\le \frac{1}{64}, \label{eq:45} \\
n(n+1)\uu &\le \frac{1}{64}, \label{eq:46} \\
11(m\uu+(n+1)\uu) \cdot kc^{2} &\le s \le \frac{1}{100n} \cdot kc^{2}, \label{eq:s1} \\ 
4.87n\sqrt{n}\uu \cdot l\kappa_{2}(X) &\le 1. \label{eq:47} 
\end{align}
Here, we define $c=\max\abs{x_{ij}}, 1 \le i \le m, 1 \le j \le n$. It is the element with the largest absolute value of the input $X \in \mathbb{R}^{m\times n}$. $k$ is the nnze of the sparse $X$ with $1 \le k \le mn$. $l=\frac{c\sqrt{k}}{\norm{X}_{2}}$.

With the settings and definitions above, we show the properties of SCholeskyQR3 with the alternative $s$ for the sparse $X \in \mathbb{R}^{m\times n}$, as shown in Theorem~\ref{THM 4.2}-Theorem~\ref{THM 4.1}. 

\begin{theorem}[The relationship between $\kappa_{2}(X)$ and $\kappa_{2}(W)$ for the sparse $X \in \mathbb{R}^{m\times n}$]
\label{THM 4.2}
If $X \in \mathbb{R}^{m\times n}$ is sparse and $[W,Y]=\mbox{SCholeskyQR}(X)$, when \eqref{eq:45} and \eqref{eq:47} are satisfied, we have
\begin{equation}
\kappa_{2}(W) \le 2h \cdot \sqrt{1+\alpha_{0}(\kappa_{2}(X))^{2}}, \label{eq:417}
\end{equation}
if $\alpha_{0}=\frac{s}{\norm{X}_{2}^{2}}=11(m\uu+(n+1)\uu) \cdot l^{2}$. For $[Q,R]=\mbox{SCholeskyQR3}(X)$ with $s=11(m\uu+(n+1)\uu) \cdot kc^{2}$, if $\kappa_{2}(X)$ is large enough, the sufficient condition of $\kappa_{2}(X)$ is 
\begin{equation}
\kappa_{2}(X) \le \frac{1}{85 \cdot l\sqrt{n} \cdot (m\uu+(n+1)\uu)}. \label{eq:418}
\end{equation}
\end{theorem}

\begin{theorem}[Rounding error analysis of SCholeskyQR3 for the sparse $X \in \mathbb{R}^{m\times n}$]
\label{THM 4.3}
With \eqref{eq:45}, \eqref{eq:46} and \eqref{eq:418}, if $X \in \mathbb{R}^{m\times n}$ is sparse and $[Q,R]=\mbox{SCholeskyQR3}(X)$, when $s=11(m\uu+(n+1)\uu) \cdot kc^{2}$, we have
\begin{align}
\norm{Q^{\top}Q-I}_{F} &\le 6(mn\uu+n(n+1)\uu), \label{eq:q} \\
\norm{QR-X}_{F} &\le (3.58l+5.45\sqrt{n}) \cdot n\sqrt{n}\uu\norm{X}_{2}. \label{eq:r}
\end{align}
\end{theorem}

We provide a new $s=11(m\uu+(n+1)\uu) \cdot kc^{2}$ in Theorem~\ref{THM 4.2} and Theorem~\ref{THM 4.3} with the corresponding rounding error analysis. Compared to the theoretical results of ISCholeskyQR3 for the general cases in \cite{New}, we focus on a common special case, sparse matrices, from a distinguished perspective. Therefore, when the input $X$ is sparse, we can make a comparison between different $s$ and take a choice of the shifted item $s$ for SCholeskyQR3 in the sparse cases. Such an optimal and alternative choice of $s$ can guarantee the accuracy and efficiency of the algorithms, as shown in Section~\ref{sec:Common} and Section~\ref{sec:experiments}.

\begin{theorem}[A choice of the shifted item $s$ for the optimal one]
\label{THM 4.1}
If $X \in \mathbb{R}^{m\times n}$ is a sparse matrix and $[Q,R]=\mbox{SCholeskyQR3}(X)$, then we can take a shifted item $s=j_{s}$, where
\begin{equation}
j_{s}=\min(11(m\uu+(n+1)\uu) \cdot kc^{2}, 11(mn\uu+n(n+1)\uu) \cdot \norm{X}_{g}^{2}). \label{eq:js1}
\end{equation}
Taking such a choice of $s$ can keep the accuracy of SCholeskyQR3 in the sparse cases.
\end{theorem}

Based on Theorem~\ref{THM 4.1}, we can have SCholeskyQR and SCholeskyQR3 for the sparse matrices in Algorithm~\ref{alg:Shiftedt1} and Algorithm~\ref{alg:Shiftedt2}. Comparisons between the theoretical results of SCholeksyQR and SCholeskyQR3 with different $s$ are listed in Table~\ref{tab:Comparisons1} and Table~\ref{tab:Comparisons2}.

\begin{algorithm}
\caption{$[Q,R]=\mbox{SCholeskyQR}(X)$ for the sparse $X \in \mathbb{R}^{m\times n}$}
\label{alg:Shiftedt1}
\begin{algorithmic}[1]
\REQUIRE $X \in \mathbb{R}^{m\times n}.$ 
\ENSURE \mbox{Orthogonal factor} $Q \in \mathbb{R}^{m\times n}$, \mbox{Upper triangular factor} $R \in \mathbb{R}^{n \times n}.$ 
\STATE get $k$ and $c$ for the input $X$, 
\STATE take $s=j_{s}$ as defined in \eqref{eq:js1}, 
\STATE $[Q,R]=\mbox{SCholeskyQR}(X).$
\end{algorithmic}
\end{algorithm}%

\begin{algorithm}
\caption{$[Q,R]=\mbox{SCholeskyQR3}(X)$ for the sparse $X \in \mathbb{R}^{m\times n}$}
\label{alg:Shiftedt2}
\begin{algorithmic}[1]
\REQUIRE $X \in \mathbb{R}^{m\times n}.$ 
\ENSURE \mbox{Orthogonal factor} $Q \in \mathbb{R}^{m\times n}$, \mbox{Upper triangular factor} $R \in \mathbb{R}^{n \times n}.$ 
\STATE get $k$ and $c$ for the input $X$, 
\STATE take $s=j_{s}$ as defined in \eqref{eq:js1}, 
\STATE $[Q,R]=\mbox{SCholeskyQR3}(X).$
\end{algorithmic}
\end{algorithm}%

\begin{table}
\caption{Comparison of $\kappa_{2}(X)$ between different $s$ when $X \in \mathbb{R}^{m\times n}$ is sparse}
\centering
\begin{tabular}{||c c c||}
\hline
$s$ & $\mbox{Sufficient condition of $\kappa_{2}(X)$}$ & $\mbox{Upper bound of $\kappa_{2}(X)$}$ \\
\hline
$11(mn\uu+n(n+1)\uu)\norm{X}_{g}^{2}$ & $\frac{1}{86p(mn\uu+n(n+1)\uu)}$ & $\frac{1}{4.89pn^{2}\uu}$ \\
\hline
$11(m\uu+(n+1)\uu) \cdot kc^{2}$ & $\frac{1}{85 \cdot l\sqrt{n} \cdot (m\uu+(n+1)\uu)}$ & $\frac{1}{4.87n\sqrt{n}\uu \cdot l}$ \\
\hline
\end{tabular}
\label{tab:Comparisons1}
\end{table}

\begin{table}
\caption{Comparison of the upper bounds of $\norm{QR-X}_{F}$ between different $s$ when $X \in \mathbb{R}^{m\times n}$ is sparse}
\centering
\begin{tabular}{||c c c||}
\hline
$s$ & $\mbox{SCholeskyQR}$ & $\mbox{SCholeskyQR3}$\\
\hline
$11(mn\uu+n(n+1)\uu)\norm{X}_{g}^{2}$ & $1.6n^{2}\uu\norm{X}_{g}$ & $(6.57p+4.87)n^{2}\uu\norm{X}_{2}$ \\
\hline
$11(m\uu+(n+1)\uu) \cdot kc^{2}$ & $1.66n\sqrt{n}\uu \cdot l\norm{X}_{2}$ & $(3.58l+5.45\sqrt{n}) \cdot n\sqrt{n}\uu\norm{X}_{2}$ \\
\hline
\end{tabular}
\label{tab:Comparisons2}
\end{table}
\begin{remark}
Theorem~\ref{THM 4.3} is one of the most important theoretical results in this work, demonstrating that when $X$ is sparse, a new $s=11(m\uu+(n+1)\uu) \cdot (vt_{1}+nt_{2})c^{2}$ can be taken for SCholeskyQR3. Theorem~\ref{THM 4.1} shows that our alternative choice of $s$ can guarantee the accuracy of the algorithm. These two theorems indicate that we can leverage the structure of the sparse $X$ to construct a new shifted item $s$, which is superior to $s=11(mn\uu+n(n+1)\uu) \cdot \norm{X}_{g}^{2}$ under some conditions. In the real implementations, it is not expensive for us to receive $k$ and $c$ using MATLAB. When our new $s$ is the optimal one, we can also prove that the residual in \eqref{eq:r} is in the same level as that in Lemma~\ref{lemma 2.10}, which will be shown in Section~\ref{sec:Common}, exhibiting the effectiveness of our choice of the shifted item $s$ for the sparse cases.
\end{remark}

\subsection{Outline of this work}
The paper is organized as follows. We show some theoretical results shown in the existing works in Section~\ref{sec:LR}, which will be used in this work. In Section~\ref{sec:Common}, we provide a theoretical analysis of SCholeskyQR3 for the sparse matrices and prove Theorem~\ref{THM 4.2}-Theorem~\ref{THM 4.1}, which constitutes the key part of this work. Following the theoretical analysis, we perform numerical experiments using typical examples from some real-world problems and present the results in Section~\ref{sec:experiments}. Finally, we summarize the results in Section~\ref{sec:conclusions}.

\section{Some lemmas regarding rounding error analysis}
\label{sec:LR}
In this section, we introduce several lemmas regarding deterministic rounding error analysis in \cite{Higham, Perturbation}. They will be used in the analysis of this work.

\begin{lemma}[Weyl's Theorem for singular values]
\label{lemma 2.1}
If $A,B \in \mathbb{R}^{m\times n}$, then
\begin{equation}
\sigma_{min}(A+B) \ge \sigma_{min}(A)-\norm{B}_{2}. \nonumber
\end{equation}
\end{lemma}

\begin{lemma}[Rounding error in matrix multiplications]
\label{lemma 2.2}
For $A \in \mathbb{R}^{m\times n}, B \in \mathbb{R}^{n\times p}$, the error in computing the matrix product $AB$ in floating-point arithmetic is bounded by
\begin{equation}
\abs{AB-fl(AB)}\le \gamma_{n}\abs{A}\abs{B}. \nonumber
\end{equation}
Here, $\abs{A}$ is the matrix whose $(i,j)$ element is $\abs{a_{ij}}$ and
\begin{equation}
\gamma_n: = \frac{n{\uu}}{1-n{\uu}} \le 1.02n{\uu}. \nonumber
\end{equation}
\end{lemma}

\begin{lemma}[Rounding error in Cholesky factorization]
\label{lemma 2.3}
If Cholesky factorization applied to the symmetric positive definite $A \in \mathbb{R}^{n\times n}$ runs to completion, then the computed factor $R \in \mathbb{R}^{n\times n}$ satisfies
\begin{equation}
R^{\top}R=A+\Delta{A}, \quad \abs{\Delta A}\le \gamma_{n+1}\abs{{R}^{\top}}\abs{R}. \nonumber
\end{equation}
\end{lemma}

\begin{lemma}[Rounding error in solving triangular systems]
\label{lemma 2.4}
Let the triangular system $Tx=b$, where $T \in \mathbb{R}^{n\times n}$ is non-singular, be solved by institution with any ordering. Then the computed solution $x$ satisfies
\begin{equation}
(T+\Delta T)x=b, \quad \abs{\Delta T}\le \gamma_{n}\abs{R}. \nonumber
\end{equation}
\end{lemma}

\section{Proof of Theorem~\ref{THM 4.2}-Theorem~\ref{THM 4.1}}
\label{sec:Common}
In this section, we prove Theorem~\ref{THM 4.2}-Theorem~\ref{THM 4.1} when $X \in \mathbb{R}^{m\times n}$ and $\mbox{rank}(X)=n$. Among all the theorems, Theorem~\ref{THM 4.2} and Theorem~\ref{THM 4.3} are the key results.

\subsection{Lemmas to prove Theorem~\ref{THM 4.2}-Theorem~\ref{THM 4.1}}
Before proving Theorem~\ref{THM 4.2}-Theorem~\ref{THM 4.1}, we write $[W,Y]=\mbox{SCholeskyQR}(X)$ with error matrices below.
\begin{align}
G=X^{\top}X &+ E_{A}, \label{eq:E1} \\
Y^{\top}Y=G &+ sI+E_{B}, \label{eq:E2} \\
w_{i}^{\top}=x_{i}^{\top}(Y &+ \Delta Y_{i})^{-1}, \label{eq:E3} \\
WY=X &+ \Delta X. \label{eq:E4}
\end{align}
Here, $x_{i}^{\top}$ and $w_{i}^{\top}$ represent the $i$-th rows of $X$ and $W$, respectively. $E_{A}$ in \eqref{eq:E1} denotes the error matrix generated when calculating the Gram matrix $X^{\top}X$. Similarly, $E_{B}$ in \eqref{eq:E2} represents the error matrix resulting from performing Cholesky factorization on $B$ with a shifted item. Since $Y$ may be non-invertible, the $w_{i}^{\top}$ can be solved by solving the linear system $(Y^{\top}+(\Delta Y_{i})^{\top})(w_{i}^{\top})^{\top}=(x_{i}^{\top})^{\top}$, that is, the transpose of \eqref{eq:E3}. We do not write this step in the form of the whole matrices because each $\Delta Y_{i}$ depends on $Y$ and $x_{i}^{\top}$, where $\Delta Y_{i}$ denotes the rounding error for the $Y$-factor when calculating $w_{i}^{\top}$. In spite of this, $\Delta Y_{i}$ has an uniform upper bound according to Lemma~\ref{lemma 2.4}. When we express the last step of Algorithm~\ref{alg:Shifted} without $Y^{-1}$, the general error matrix of QR factorization is given by $\Delta X$ in \eqref{eq:E4}. 

The general idea of the theoretical analysis to prove Theorem~\ref{THM 4.1} is similar to that in \cite{New, Shifted, error}. However, we make some improvements in the proof, building connections between the sparsity and CholeskyQR. Some lemmas are shown below in order to prove Theorem~\ref{THM 4.2}-Theorem~\ref{THM 4.1}.

\begin{lemma}[Estimating $\norm{E_{A}}_{2}$ and $\norm{E_{B}}_{2}$]
\label{lemma 2.20}
For $\norm{E_{A}}_{2}$ and $\norm{E_{B}}_{2}$ in \eqref{eq:E1} and \eqref{eq:E2}, when \eqref{eq:s1} is satisfied, we have
\begin{align}
\norm{E_{A}}_{2} &\le 1.1m\uu \cdot kc^{2}, \label{eq:413} \\
\norm{E_{B}}_{2} &\le 1.1(n+1)\uu \cdot kc^{2}. \label{eq:414}
\end{align}
\end{lemma}
\begin{proof}
For $E_{A}$ in \eqref{eq:E1}, with the settings for the sparse $X$ before, we can bound $\norm{X}_{F}$ as $\norm{X}_{F} \le kc^{2}$. Therefore, according to Lemma~\ref{lemma 2.2}, we can estimate $\norm{E_{A}}_{2}$ as
\begin{equation}
\begin{split}
\norm{E_{A}}_{2} &\le \norm{E_{A}}_{F} \nonumber \\ &\le \gamma_{m} \cdot \norm{X}_{F}^{2} \nonumber \\ &\le 1.1m\uu \cdot kc^{2}. \nonumber
\end{split}
\end{equation}
\eqref{eq:413} is proved.

For $\norm{E_{B}}_{2}$, with Lemma~\ref{lemma 2.3}, \eqref{eq:E1} and \eqref{eq:E2}, we can get
\begin{equation}
\norm{E_{B}}_{2} \le \norm{\abs{E_{B}}}_{F} \le \gamma_{n+1}\norm{\abs{Y}}_{F}^{2}. \label{eq:yf2}
\end{equation}
In fact, $\norm{Y}_{F}^{2}$ equals to $\tr(Y^{\top}Y)$, which denotes the trace of the gram matrix $Y^{\top}Y$. With the setting of the sparse $X$, \eqref{eq:E1} and \eqref{eq:E2}, we can get
\begin{equation} \label{eq:try}
\begin{split}
\gamma_{n+1}\norm{Y}_{F}^{2} &\le \gamma_{n+1}\tr(Y^{\top}Y) \\ &\le \gamma_{n+1}\tr(X^{\top}X+sI+E_{A}+E_{B}) \\ &\le \gamma_{n+1}(\norm{X}_{F}^{2}+sn+n\norm{E_{A}}_{2}+n\norm{E_{B}}_{2}) \\ &\le \gamma_{n+1} \cdot (kc^{2}+sn+n\norm{E_{A}}_{F}+n\norm{E_{B}}_{F}). 
\end{split}
\end{equation}
We combine \eqref{eq:yf2} and \eqref{eq:try} with \eqref{eq:45}, \eqref{eq:46}, \eqref{eq:s1} and \eqref{eq:413}, and we can bound $\norm{E_{B}}_{2}$ as
\begin{equation}
\begin{split}
\norm{E_{B}}_{2} &\le \frac{\gamma_{n+1} \cdot (1+1.1mn\uu+wn)}{1-\gamma_{n+1} \cdot n} \cdot kc^{2} \nonumber \\
&\le \frac{1.02(n+1){\uu} \cdot (1+1.1mn\uu+0.01)}{1-1.02(n+1)\uu \cdot n} \cdot kc^{2} \nonumber \\
&\le \frac{1.02(n+1)\uu\cdot(1+1.1\cdot\frac{1}{64}+0.01)}{1-\frac{1.02}{64}} \cdot kc^{2} \nonumber \\
&\le 1.1(n+1)\uu \cdot kc^{2}. \nonumber
\end{split}
\end{equation}
Here, $w=\frac{s}{kc^{2}}=11(m\uu+(n+1)\uu)$. \eqref{eq:414} is proved.
\end{proof}
\begin{remark}
The steps to prove \eqref{eq:414} contain a step utilizing the properties of the traces in \eqref{eq:try}. This idea of proof has not occurred in the works of CholeskyQR before. Although \eqref{eq:414} seems similar to the corresponding results in \cite{New, Shifted, error}, our ideas in the theoretical analysis are distinguished from those in the previous works.
\end{remark}

\begin{lemma}[Estimating $\norm{Y^{-1}}_{2}$ and $\norm{XY^{-1}}_{2}$]
\label{lemma 2.11}
For $\norm{Y^{-1}}_{2}$ and $\norm{XY^{-1}}_{2}$ in \eqref{eq:E3}, when \eqref{eq:s1} is satisfied, we have
\begin{align}
\norm{Y^{-1}}_{2} &\le \frac{1}{\sqrt{(\sigma_{min}(X))^{2}+0.9s}}, \label{eq:L1} \\
\norm{XY^{-1}}_{2} &\le 1.5. \label{eq:L2}
\end{align}
\end{lemma}
\begin{proof}
The steps to prove \eqref{eq:L1} and \eqref{eq:L2} are the same as those in \cite{New, Shifted}.
\end{proof}

\begin{lemma}[Estimating $\norm{\Delta Y_{i}}_{2}$]
\label{lemma 2.12}
For $\norm{\Delta Y_{i}}_{2}$ in \eqref{eq:E3}, when \eqref{eq:s1} is satisfied, we have
\begin{equation}
\norm{\Delta Y_{i}}_{2} \le 1.03n\uu \cdot c\sqrt{k}. \label{eq:L3}
\end{equation}
\end{lemma}
\begin{proof}
Regarding \eqref{eq:E3}, based on Lemma~\ref{lemma 2.4}, we can have
\begin{equation} 
\norm{\Delta Y_{i}}_{2} \le \gamma_{n} \cdot \norm{Y}_{F} \le 1.02n\uu \cdot \norm{Y}_{F}. \label{eq:434}
\end{equation}
With \eqref{eq:s1}, \eqref{eq:413} and \eqref{eq:414}, we can have
\begin{equation} \label{eq:abs}
\begin{split}
\norm{E_{A}}_{2}+\norm{E_{B}}_{2} &\le 1.1(m\uu+(n+1)\uu) \cdot kc^{2} \\ &\le 0.1s. 
\end{split}
\end{equation}
For $\norm{Y}_{F}$, similar to \eqref{eq:try} and based on \eqref{eq:s1}, \eqref{eq:E1} and \eqref{eq:E2}, we can have
\begin{equation} \label{eq:432}
\begin{split}
\norm{Y}_{F}^{2} &\le \norm{X}_{F}^{2}+sn+n \cdot (\norm{E_{A}}_{2}+\norm{E_{B}}_{2}) \\ &\le kc^{2}+1.1sn \\ &\le 1.011kc^{2}. 
\end{split}
\end{equation}
Therefore, with \eqref{eq:432}, it is easy to see that
\begin{equation}
\norm{Y}_{F} \le 1.006c\sqrt{k}. \label{eq:433}
\end{equation}
We put \eqref{eq:433} into \eqref{eq:434} and we can have \eqref{eq:L3}. \eqref{eq:L3} is proved.
\end{proof}

\begin{lemma}[Estimating $\norm{\Delta X}_{2}$]
\label{lemma 2.13}
For $\norm{\Delta X}_{2}$ in \eqref{eq:E4}, when \eqref{eq:s1} is satisfied, we have
\begin{equation}
\norm{\Delta X}_{2} \le \frac{1.09n\uu \cdot kc^{2}}{\sqrt{(\sigma_{min}(X))^{2}+0.9s}}. \label{eq:L4}
\end{equation}
\end{lemma}
\begin{proof}
Similar to the approach in \cite{New, Shifted}, we can express \eqref{eq:E3} as
\begin{equation}
w_{i}^{\top}=x_{i}^{\top}(Y+\Delta Y_{i})^{-1}=x_{i}^{\top}(I+Y^{-1}\Delta Y_{i})^{-1}Y^{-1}. \label{eq:217}
\end{equation}
When we define
\begin{equation}
(I+Y^{-1}\Delta Y_{i})^{-1}=I+ \theta_{i}, \label{eq:218}
\end{equation}
where
\begin{equation}
\theta_{i}:=\sum_{j=1}^{\infty}(-Y^{-1}\Delta Y_{i})^{j}, \label{eq:219}
\end{equation}
based on \eqref{eq:E3} and \eqref{eq:E4}, we can have
\begin{equation}
\Delta{x_{i}}^{\top}=x_{i}^{\top}\theta_{i}. \label{eq:220}
\end{equation}
$\Delta x_{i}$ is the $i$-th row of $\Delta X$. Based on \eqref{eq:46}, \eqref{eq:s1}, \eqref{eq:L1} and \eqref{eq:L3}, when \eqref{eq:s1} is satisfied, we can have
\begin{equation} \label{eq:435}
\begin{split}
\norm{Y^{-1}\Delta Y_{i}}_{2} &\le \norm{Y^{-1}}_{2}\norm{\Delta Y_{i}}_{2} \\ &\le \frac{1.03n\uu \cdot c\sqrt{k}}{\sqrt{(\sigma_{min}(X))^{2}+0.9s}} \\ &\le \frac{1.03n\uu \cdot c\sqrt{k}}{\sqrt{0.9s}} \\ &\le \frac{1.03n\uu \cdot c\sqrt{k}}{\sqrt{9.9(m\uu+(n+1)\uu) \cdot kc^{2}}} \\ &\le \frac{1.03}{\sqrt{9.9}} \cdot n\sqrt{\uu} \\ &\le 0.05. 
\end{split}
\end{equation}
For \eqref{eq:219}, with \eqref{eq:L1}, \eqref{eq:L3} and \eqref{eq:435}, we can have
\begin{equation} \label{eq:436}
\begin{split}
\norm{\theta_{i}}_{2} &\le \sum_{j=1}^{\infty}(\norm{Y^{-1}}_{2}\norm{\Delta Y_{i}}_{2})^{j} \\ &= \frac{\norm{Y^{-1}}_{2}\norm{\Delta Y_{i}}_{2}}{1-\norm{Y^{-1}}_{2}\norm{\Delta Y_{i}}_{2}} \\ &\le \frac{1}{0.95} \cdot \frac{1.03n\uu \cdot c\sqrt{k}}{\sqrt{(\sigma_{min}(X))^{2}+0.9s}} \\ &\le \frac{1.09n\uu \cdot c\sqrt{k}}{\sqrt{(\sigma_{min}(X))^{2}+0.9s}}. 
\end{split}
\end{equation}
Based on \eqref{eq:220}, it is easy to see that
\begin{equation}
\norm{\Delta  x_{i}^{\top}}_{2} \le \norm{x_{i}^{\top}}_{2}\norm{\theta_{i}}_{2}. \label{eq:437}
\end{equation}
Therefore, similar to the step in \cite{Shifted}. together with the setting of the sparse $X$ and \eqref{eq:437}, we can have
\begin{equation} \label{eq:delta1}
\begin{split}
\norm{\Delta X}_{2} &\le \norm{\Delta X}_{F} \\ &\le \norm{X}_{F}\norm{\theta_{i}}_{2} \\ &\le c\sqrt{k} \cdot \norm{\theta_{i}}_{2}. 
\end{split}
\end{equation}
We put \eqref{eq:437} into \eqref{eq:delta1} and we can have \eqref{eq:L4}. \eqref{eq:L4} is proved.
\end{proof}

\subsection{Proof of Theorem~\ref{THM 4.2}}
Here, we prove Theorem~\ref{THM 4.2} with Lemma~\ref{lemma 2.11}-Lemma~\ref{lemma 2.13}. 

\begin{proof}
The general approach to proving Theorem~\ref{THM 4.2} is similar to those in \cite{New, Shifted}. However, we establish connections between the sparsity of $X$ and QR factorization. Our proof will be divided into three parts: estimating $\norm{W^{\top}W-I}_{F}$, estimating $\norm{\Delta X}_{F}$, and analyzing the relationship between $\kappa_{2}(X)$ and $\kappa_{2}(W)$.

\subsubsection{Estimating $\norm{W^{\top}W-I}_{2}$}
With \eqref{eq:E1}-\eqref{eq:E4}, we can have
\begin{equation}
\begin{split}
W^{\top}W &= Y^{-\top}(X+\Delta X)^{\top}(X+\Delta X)Y^{-1} \nonumber \\ 
&= Y^{-\top}X^{\top}XY^{-1}+Y^{-\top}X^{\top}\Delta XY^{-1} \nonumber \\ 
&+ Y^{-\top}\Delta X^{\top}XY^{-1}+Y^{-\top}\Delta X^{\top}\Delta XY^{-1} \nonumber \\
&= I-Y^{-\top}(sI+E_{1}+E_{2})Y^{-1}+(XY^{-1})^{\top}\Delta XY^{-1} \nonumber \\ 
&+ Y^{-\top}\Delta X^{\top}(XY^{-1})+Y^{-\top}\Delta X^{\top}\Delta XY^{-1}. \nonumber
\end{split}
\end{equation}
Therefore, we can have
\begin{equation} \label{eq:440}
\begin{split}
\norm{W^{\top}W-I}_{2} &\le \norm{Y^{-1}}_{2}^{2}(\norm{E_{A}}_{2}+\norm{E_{B}}_{2}+s)+2\norm{Y^{-1}}_{2}\norm{XY^{-1}}_{2}\norm{\Delta X}_{2} \\ &+ \norm{Y^{-1}}_{2}^{2}\norm{\Delta X}_{2}^{2}. 
\end{split}
\end{equation}
According to \eqref{eq:L1} and \eqref{eq:abs}, we can have
\begin{equation} \label{eq:441} 
\begin{split}
\norm{Y^{-1}}_{2}^{2}(\norm{E_{A}}_{2}+\norm{E_{B}}_{2}+s) &\le \frac{1.1s}{(\sigma_{min}(X))^{2}+0.9s} \\ &\le 1.23. 
\end{split}
\end{equation}
Based on \eqref{eq:L1}, \eqref{eq:L2} and \eqref{eq:L4}, we can have
\begin{equation} \label{eq:442}
\begin{split}
2\norm{Y^{-1}}_{2}\norm{XY^{-1}}_{2}\norm{\Delta X}_{2} &\le 2 \cdot \frac{1}{\sqrt{(\sigma_{min}(X))^{2}+0.9s}} \cdot 1.5 \\ &\cdot \frac{1.09n\uu \cdot kc^{2}}{\sqrt{(\sigma_{min}(X))^{2}+0.9s}} \\ &\le \frac{3.27n\uu \cdot kc^{2}}{(\sigma_{min}(X))^{2}+0.9s} \\ &\le \frac{3.27 n\uu \cdot kc^{2}}{9.9(m\uu+(n+1)\uu) \cdot kc^{2}} \\ &\le 0.34. 
\end{split}
\end{equation}
With \eqref{eq:L1} and \eqref{eq:L4}, we can have
\begin{equation} \label{eq:443}
\begin{split}
\norm{Y^{-1}}_{2}^{2}\norm{\Delta X}_{2}^{2} &\le \frac{1}{(\sigma_{min}(X))^{2}+0.9s} \cdot \frac{(1.09n\uu \cdot kc^{2})^{2}}{(\sigma_{min}(X))^{2}+0.9s} \\ &\le \frac{(1.09n\uu \cdot kc^{2})^{2}}{[9.9(m\uu+(n+1)\uu) \cdot kc^{2}]^{2}} \\ &\le 0.013. 
\end{split}
\end{equation}
Therefore, we put \eqref{eq:441}-\eqref{eq:443} into \eqref{eq:440} and we can have
\begin{equation} \label{eq:444} 
\begin{split}
\norm{W^{\top}W-I}_{2} &\le 1.23+0.34+0.013 \\ &\le 1.59.
\end{split}
\end{equation}
With \eqref{eq:444}, we can have
\begin{equation}
\norm{W}_{2} \le 1.61. \label{eq:446}
\end{equation}

\subsubsection{Estimating $\norm{\Delta X}_{F}$}
Regarding $\norm{\Delta X}_{F}$ in \eqref{eq:E4}, similar to the results in \cite{New, Shifted}, together with \eqref{eq:L3} and \eqref{eq:446}, we can have
\begin{equation} \label{eq:447}
\begin{split}
\norm{\Delta X}_{F} &= \norm{WY-X}_{F} \\ &\le \norm{W}_{F} \cdot \norm{\Delta Y_{i}}_{2} \\ &\le 1.61\sqrt{n} \cdot 1.03n\uu \cdot c\sqrt{k} \\ &\le 1.66n\sqrt{n}\uu \cdot c\sqrt{k} \\ &= 1.66n\sqrt{n}\uu \cdot l\norm{X}_{2}. 
\end{split}
\end{equation}
Here, $l=\frac{c\sqrt{k}}{\norm{X}_{2}}$. 

\subsubsection{The relationship between $\kappa_{2}(X)$ and $\kappa_{2}(W)$}
In order to estimate $\kappa_{2}(W)$, since we have already estimated $\norm{W}_{2}$, we only need to estimate $\sigma_{min}(W)$. Based on Lemma~\ref{lemma 2.1}, we can have
\begin{equation}
\sigma_{min}(W) \ge \sigma_{min}(XY^{-1})-\norm{\Delta XY^{-1}}_{2}. \label{eq:448}
\end{equation}
With \eqref{eq:L1} and \eqref{eq:447}, we can have
\begin{equation}
\norm{\Delta XY^{-1}}_{2} \le \norm{\Delta X}_{2}\norm{Y^{-1}}_{2} \le \frac{1.66n\sqrt{n}\uu \cdot l\norm{X}_{2}}{\sqrt{(\sigma_{min}(X))^{2}+0.9s}}. \label{eq:449}
\end{equation}
The same as the result in \cite{Shifted}, we can have
\begin{equation}
\sigma_{min}(XY^{-1}) \ge \frac{\sigma_{min}(X)}{\sqrt{(\sigma_{min}(X))^{2}+s}} \cdot 0.9. \label{eq:450}
\end{equation}
Therefore, based on \eqref{eq:47}, we put \eqref{eq:449} and \eqref{eq:450} into \eqref{eq:448} and we can have
\begin{equation} \label{eq:451} 
\begin{split}
\sigma_{min}(W) &\ge \frac{0.9\sigma_{min}(X)}{\sqrt{(\sigma_{min}(X))^{2}+s}}-\frac{1.66n\sqrt{n}\uu \cdot l\norm{X}_{2}}{\sqrt{(\sigma_{min}(X))^{2}+0.9s}} \\ &\ge \frac{0.9}{\sqrt{(\sigma_{min}(X))^{2}+s}}(\sigma_{min}(X)-\frac{1.66}{0.9 \cdot \sqrt{0.9}} \cdot n\sqrt{n}\uu \cdot l\norm{X}_{2}) \\ &\ge \frac{\sigma_{min}(X)}{2\sqrt{(\sigma_{min}(X))^{2}+s}} \\ &= \frac{1}{2\sqrt{1+\alpha_{0}(\kappa_{2}(X))^{2}}}, 
\end{split}
\end{equation}
where $\alpha_{0}=\frac{s}{\norm{X}_{2}^{2}}$. With \eqref{eq:446} and \eqref{eq:451}, we can have
\begin{equation}
\kappa_{2}(W) \le 3.22 \cdot \sqrt{1+\alpha_{0}(\kappa_{2}(X))^{2}}. \nonumber
\end{equation}
Therefore, \eqref{eq:417} is proved. 

If we take $s=11(m\uu+(n+1)\uu) \cdot kc^{2}=11(m\uu+(n+1)\uu) \cdot l^{2}$, with the setting of the sparse $X$, we can have $\alpha_{0}=11(m\uu+(n+1)\uu) \cdot \frac{kc^{2}}{\norm{X}_{2}^{2}}=11(m\uu+(n+1)\uu) \cdot l^{2} \ge 11(m\uu+(n+1)\uu)$. When $\kappa_{2}(X)$ is large, \textit{e.g.}, $\kappa_{2}(X) \ge \uu^{-\frac{1}{2}}$, $\alpha_{0}(\kappa_{2}(X))^{2} \ge 11(m+n)>>1$. Therefore, we can have
\begin{equation}
3.22 \cdot \sqrt{1+\alpha_{0}(\kappa_{2}(X))^{2}} \approx 3.22 \cdot \sqrt{\alpha_{0}} \cdot \kappa_{2}(X). \nonumber
\end{equation}
So it is easy to see that 
\begin{equation}
\kappa_{2}(W) \le 3.22 \cdot \sqrt{\alpha_{0}} \cdot \kappa_{2}(X). \label{eq:relation}
\end{equation}
Using the similar method as that in \cite{New, Shifted}, in order to receive a sufficient condition for SCholeskyQR3, we let
\begin{equation}
\kappa_{2}(W) \le 3.22 \cdot \sqrt{\alpha_{0}} \cdot \kappa_{2}(X) \le \frac{1}{8\sqrt{mn\uu+n(n+1)\uu}}. \label{eq:q2} 
\end{equation}
We put $\alpha_{0}=\frac{s}{\norm{X}_{2}^{2}}=11(m\uu+(n+1)\uu) \cdot l^{2}$ into \eqref{eq:q2} and we can have \eqref{eq:418}. \eqref{eq:418} is proved. Therefore, Theorem~\ref{THM 4.2} is proved.
\end{proof}

\subsection{Proof of Theorem~\ref{THM 4.3}}
In this section, we prove Theorem~\ref{THM 4.3}. Our method to prove Theorem~\ref{THM 4.3} is similar to that in \cite{New}.

\begin{proof}
When $s=11(m\uu+(n+1)\uu) \cdot kc^{2}$ and $\kappa_{2}(X)$ satisfies \eqref{eq:418}, we can easily derive \eqref{eq:q} with $\kappa_{2}(X)$, which is the same as that in \cite{error}.

For the residual of SCholeskyQR3, $\norm{QR-X}_{F}$, we express the CholeskyQR2 after SCholeskyQR with the error matrices as follows.
\begin{align}
C-W^{\top}W &= E_{1}, \nonumber \\
D^{\top}D-C &= E_{2}, \nonumber \\
VD-W &= E_{3}, \label{eq:457} \\
DY-N &= E_{4}. \label{eq:458} \\
B-V^{\top}V &= E_{5}, \nonumber \\
J^{\top}J-B &= E_{6}, \nonumber \\
QJ-V &= E_{7}, \label{eq:461} \\
JN-R &= E_{8}. \label{eq:462}
\end{align}
Here, the calculation of $R$ in Algorithm~\ref{alg:Shiftedt2} is divided into two steps, that is, \eqref{eq:458} and \eqref{eq:462}. Similar to that of  \cite{error}, $Z$ in Algorithm~\ref{alg:Shiftedt2} satisfies $Z=JD$ without error matrices. Therefore, $R=ZY$ should be written as $R=JDY$ without considering rounding errors. In order to simplify rounding error analysis of this step, we write the multiplication of $D$ and $Y$ with error matrices as \eqref{eq:458} and the multiplication of $J$ and $N$ can be written as \eqref{eq:462}. Based on \eqref{eq:457}-\eqref{eq:462}, we can have
\begin{equation} \label{eq:Q2R4}
\begin{split}
QR &= (V+E_{7})J^{-1}(JN-E_{8}) \\ &= (V+E_{7})N-(V+E_{7})J^{-1}E_{8} \\ &= VN+E_{7}N-QE_{8} \\ &= (W+E_{3})D^{-1}(DY-E_{4})+E_{7}N-QE_{8} \\ &= (W+E_{3})Y-(W+E_{3})D^{-1}E_{4}+E_{7}N-QE_{8} \\ &= WY+E_{3}Y-VE_{4}+E_{7}N-QE_{8}. 
\end{split}
\end{equation}
Therefore, based on \eqref{eq:Q2R4}, we can get
\begin{equation} \label{eq:463}
\begin{split}
\norm{QR-X}_{F} &\le \norm{WY-X}_{F}+\norm{E_{3}}_{F}\norm{Y}_{2}+\norm{V}_{2}\norm{E_{4}}_{F} \\ &+ \norm{E_{7}}_{F}\norm{N}_{2}+\norm{Q}_{2}\norm{E_{8}}_{F}. 
\end{split}
\end{equation}
Similar to \eqref{eq:E3}, we rewrite \eqref{eq:457} through rows as 
$v_{i}^{\top}=w_{i}^{\top}(D+ \Delta D_{i})^{-1}$ where $v_{i}^{\top}$ and $w_{i}^{\top}$ represent the $i$-th rows of $V$ and $W$. Similar to \eqref{eq:432} and with \eqref{eq:446}, we can bound $\norm{\Delta D_{i}}_{2}$ as
\begin{equation} \label{eq:466} 
\begin{split}
\norm{\Delta D_{i}}_{2} &\le 1.02n\uu \cdot \norm{D}_{F} \\ &\le 1.02n\uu \cdot 1.001\norm{W}_{F} \\ &\le 1.02n\sqrt{n}\uu \cdot 1.001\norm{W}_{2} \\ &\le 1.65n\sqrt{n}\uu. 
\end{split}
\end{equation}
Similar to the results in \cite{error} and with \eqref{eq:446}, $\norm{D}_{2}$ can be bounded as
\begin{equation} \label{eq:468} 
\begin{split}
\norm{D}_{2} &\le 1.001\norm{W}_{2} \\ &\le 1.62. 
\end{split}
\end{equation}
The same as that in \cite{Shifted, error}, we can get
\begin{align}
\norm{Y}_{2} &\le 1.006\norm{X}_{2}, \label{eq:464} \\
\norm{V}_{2} &\le \frac{\sqrt{69}}{8}. \label{eq:467} 
\end{align}
When $l=\frac{c\sqrt{t_{1}}}{\norm{X}_{2}}$, together with Lemma~\ref{lemma 2.2}, \eqref{eq:433} and \eqref{eq:466}-\eqref{eq:467}, we can bound $\norm{E_{3}}_{F}$ and $\norm{E_{4}}_{F}$ as
\begin{equation} \label{eq:469}
\begin{split}
\norm{E_{3}}_{F} &\le \norm{V}_{F} \cdot \norm{\Delta D_{i}}_{2} \\ &\le \frac{\sqrt{69n}}{8} \cdot 1.65n\sqrt{n}\uu \\ &\le 1.72n^{2}\uu, 
\end{split}
\end{equation}
\begin{equation} \label{eq:470}
\begin{split}
\norm{E_{4}}_{F} &\le \gamma_{n} \cdot (\norm{D}_{F} \cdot \norm{Y}_{F}) \\ &\le \gamma_{n} \cdot (\sqrt{n}\norm{D}_{2} \cdot \sqrt{n}\norm{Y}_{2}) \\ &\le 1.02n\uu \cdot 1.62\sqrt{n} \cdot 1.006\sqrt{n}\norm{X}_{2} \\ &\le 1.67n^{2}\uu\norm{X}_{2}. 
\end{split}
\end{equation}
Moreover, with \eqref{eq:46}, \eqref{eq:468}, \eqref{eq:464} and \eqref{eq:470}, $\norm{N}_{F}$ can be bounded as
\begin{equation} \label{eq:471}
\begin{split}
\norm{N}_{2} &\le \norm{D}_{2}\norm{Y}_{2}+\norm{E_{4}}_{2} \\ &\le 1.62 \cdot 1.006\norm{X}_{2}+1.67n^{2}\uu\norm{X}_{2} \\ &= 1.66\norm{X}_{2},  
\end{split}
\end{equation}
\begin{equation} \label{eq:472}
\begin{split}
\norm{N}_{F} &\le \norm{D}_{2}\norm{Y}_{F}+\norm{E_{4}}_{F} \\ &\le 1.62 \cdot 1.006c\sqrt{k}+1.67n^{2}\uu\norm{X}_{2} \\ &= (1.63l+1.67n^{2}\uu) \cdot \norm{X}_{2},  
\end{split}
\end{equation}
If we rewrite \eqref{eq:461} through rows as $q_{i}^{\top}=v_{i}^{\top}(J+\Delta J_{i})^{-1}$ where $q_{i}^{\top}$ and $v_{i}^{\top}$ represent the $i$-th rows of $Q$ and $V$, similar to \eqref{eq:466} and with \eqref{eq:467}, we can bound $\norm{\Delta J_{i}}_{2}$ as 
\begin{equation} \label{eq:475} 
\begin{split}
\norm{\Delta J_{i}}_{2} &\le 1.02n\uu \cdot \norm{V}_{F} \\ &\le 1.02n\sqrt{n}\uu \cdot \norm{V}_{2} \\ &\le 1.02n\sqrt{n}\uu \cdot \frac{\sqrt{69}}{8} \\ &\le 1.06n\sqrt{n}\uu. 
\end{split}
\end{equation}
Similar to \eqref{eq:468} and with \eqref{eq:467}, we can bound $\norm{J}_{2}$ as
\begin{equation} \label{eq:476}
\begin{split}
\norm{J}_{2} &\le 1.001\norm{V}_{2} \\ &\le 1.04. 
\end{split}
\end{equation}
According to the corresponding results in \cite{Shifted, error}, we can get
\begin{equation}
\norm{Q}_{2} \le 1.1, \label{eq:474} 
\end{equation}
With Lemma~\ref{lemma 2.2}, \eqref{eq:46} and \eqref{eq:472}-\eqref{eq:474}, we can bound $\norm{E_{7}}_{F}$ and $\norm{E_{8}}_{F}$ as
\begin{equation} \label{eq:477} 
\begin{split}
\norm{E_{7}}_{F} &\le \norm{Q}_{F} \cdot \norm{\Delta J_{i}}_{2}, \\ &\le 1.1\sqrt{n} \cdot 1.06n\sqrt{n}\uu \\ &\le 1.17n^{2}\uu, 
\end{split}
\end{equation}
\begin{equation} \label{eq:478}
\begin{split}
\norm{E_{8}}_{F} &\le \gamma_{n} \cdot (\norm{J}_{F} \cdot \norm{N}_{F}) \\ &\le \gamma_{n} \cdot (\sqrt{n}\norm{J}_{2} \cdot \norm{N}_{F}) \\ &\le 1.02n\uu \cdot 1.04\sqrt{n} \cdot (1.63l+1.67n^{2}\uu)\norm{X}_{2} \\ &\le 1.07n\sqrt{n}\uu \cdot (1.63l+0.03) \cdot \norm{X}_{2}. 
\end{split}
\end{equation}
Therefore, we put \eqref{eq:447}, \eqref{eq:464}-\eqref{eq:471} and \eqref{eq:474}-\eqref{eq:478} into \eqref{eq:463} and we can have \eqref{eq:r}. Theorem~\ref{THM 4.3} is proved.
\end{proof}

\subsection{Proof of Theorem~\ref{THM 4.1}}
In this part, we prove Theorem~\ref{THM 4.1} based on the connection between Theorem~\ref{THM 4.3} and Lemma~\ref{lemma 2.10} from \cite{New}.

\begin{proof}
According to Lemma~\ref{lemma 2.10} and Theorem~\ref{THM 4.3}, we can find that SCholeskyQR3 has the same upper bound of orthogonality with different $s$. Regarding the residual, when we make comparison between $s=11(m\uu+(n+1)\uu) \cdot kc^{2}$ and $s=11(mn\uu+(n+1)\uu) \cdot \norm{X}_{g}^{2}$, we find that when $11(m\uu+(n+1)\uu) \cdot kc^{2} \le 11(mn\uu+n(n+1)\uu) \cdot \norm{X}_{g}^{2}$, $c \le \sqrt{\frac{n\norm{X}_{g}^{2}}{k}}$. Therefore, $l=\frac{c\sqrt{k}}{\norm{X}_{2}} \le p\sqrt{n}$ with $p=\frac{\norm{X}_{g}}{\norm{X}_{2}}$. With such a condition, we can see that the residual in \eqref{eq:r} is in the same level as the case with the original $s$, which guarantees the accuracy when our alternative $s$ is the optimal one compared to the original $s$ in \cite{New}. Therefore, we can take a shifted item $s$ in the form of $s=j_{s}$, where $j_{s}=\min(11(m\uu+(n+1)\uu) \cdot kc^{2}, 11(mn\uu+n(n+1)\uu) \cdot \norm{X}_{g}^{2})$. \eqref{eq:js1} is proved. Therefore, Theorem~\ref{THM 4.1} is proved.
\end{proof}
\begin{remark}
Among all the lemmas used to prove Theorem~\ref{THM 4.2}-Theorem~\ref{THM 4.1}, Lemma~\ref{lemma 2.20} is very important. Our alternative $s$ is based on \eqref{eq:413} and \eqref{eq:414} with our settings of the properties of the sparse $X$. The proof of Lemma~\ref{lemma 2.20} lays a solid foundation for the subsequent analysis. For the sparse $X$, estimating $\norm{X}_{2}$ through the element and structure of $X$ is challenging. We often need to estimate $\norm{X}_{F}$ to replace $\norm{X}_{2}$, which brings improvements on the analysis in this work, \textit{e.g.}, proofs of Lemma~\ref{lemma 2.12}, Lemma~\ref{lemma 2.13} and Theorem~\ref{THM 4.3}. Although we do not calculate $\norm{X}_{F}$ directly in the proof of these theorems, its connection to the structure and the element of $X$ greatly simplifies our analysis, leveraging the relationship between the sparsity of the input $X$ and CholeskyQR-type algorithms. 
\end{remark}

\section{Numerical experiments}
\label{sec:experiments}
In this section, we conduct numerical experiments to examine the properties of SCholeskyQR3 for sparse matrices. We primarily focus on the applicability, accuracy, and CPU time (s) of SCholeskyQR3 with the optimal $s$ in \eqref{eq:js1}. The experiments are performed on our own laptop using MATLAB R2022a, and the specifications of the computer are listed in Table~\ref{tab:1}. For clarity, we refer to $s=j_{s}=\min(11(m\uu+(n+1)\uu) \cdot kc^{2}, 11(mn\uu+n(n+1)\uu) \cdot \norm{X}_{g}^{2})$ in Theorem~\ref{THM 4.1} as 'Alternative', while $s=11(mn\uu+n(n+1)\uu) \cdot \norm{X}_{g}^{2}$ in \cite{New} is referred to as 'Original'.

\begin{table}
\begin{center}
\centering
\caption{The specifications of our computer}
\begin{tabular}{c|c}
\hline
Item & Specification\\
\hline \hline
System & Windows 11 family(10.0, Version 22000) \\
BIOS & GBCN17WW \\
CPU & Intel(R) Core(TM) i5-10500H CPU @ 2.50GHz  -2.5 GHz \\
Number of CPUs / node & 12 \\
Memory size / node & 8 GB \\ 
Direct Version & DirectX 12 \\
\hline
\end{tabular}
\label{tab:1}
\end{center}
\end{table}

\subsection{The applicability and accuracy of the algorithm}
In this part, we do numerical experiments to show the applicability and accuracy of SCholeskyQR3 with our choice of $s$ for the sparse matrices. We take several different examples to demonstrate the effectiveness of such a choice of $s$ for the sparse cases compared to the algorithm with the original $s$ in \cite{New}.

\subsubsection{Example I: the matrix with dense columns}
In the real applications, sparse matrices with dense columns are very common in graph theory, control theory, and certain eigenvalue problems, see \cite{Constructing, Li, Eigen} and their references. One of the most well-known $T_{1}$ matrices is the arrowhead matrix, which features a dense column and a dense row. In this part, we take a arrowhead matrix $X \in \mathbb{R}^{2048\times 64}$. We define some vectors in the beginning: $e_{1ns}=(1,0,0, \cdots, 0,0)^{\top} \in \mathbb{R}^{64}$, $e_{1zs}=(0,1,1, \cdots, 1,1)^{\top} \in \mathbb{R}^{64}$, $e_{1nb}=(1,0,0, \cdots, 0,0)^{\top} \in \mathbb{R}^{2048}$ and $e_{1zs}=(0,1,1, \cdots, 1,1)^{\top} \in \mathbb{R}^{2048}$, together with a diagonal matrix $E={\rm diag}(1, \beta^{\frac{1}{63}}, \cdots, \beta^{\frac{62}{63}}, \beta) \in \mathbb{R}^{64\times 64}$. Moreover, a large matrix $\mathbb{O}_{1984\times 64}$ is formed with all the elements $0$. Therefore, a matrix $P_{sparse} \in \mathbb{R}^{2048\times 64}$ is formed as
\begin{equation}
P_{sparse}=
\begin{pmatrix}
E \\
\mathbb{O}_{1984\times 64} \nonumber
\end{pmatrix}.
\end{equation} 
We build $X \in \mathbb{R}^{2048\times 64}$ as
\begin{equation}
X=-5e_{1nb} \cdot e_{1zs}^{\top}-10e_{1zs} \cdot e_{1ns}^{\top}+P_{sparse}
\end{equation}

As a comparison group, we construct a common dense matrix $U$ using the same method described in \cite{New, Shifted, error}. $U$ is constructed using Singular Value Decomposition (SVD), and we control $\kappa_{2}(U)$ through $\sigma_{min}(U)$. We let $U \in \mathbb{R}^{2048\times 64}$
\begin{equation}
U=O \Sigma H^{\top}. \nonumber 
\end{equation}
Here, $O \in \mathbb{R}^{2048\times 2048}, H \in \mathbb{R}^{64\times 64}$ are random orthogonal matrices and
\begin{equation}
\Sigma = {\rm diag}(1, \sigma^{\frac{1}{63}}, \cdots, \sigma^{\frac{62}{63}}, \sigma) \in \mathbb{R}^{2048\times 64} \nonumber
\end{equation}
is a diagonal matrix. Here, $0<\sigma=\sigma_{min}(U)<1$ is a constant. Therefore, we have $\sigma_{1}(U)=\norm{U}_{2}=1$ and $\kappa_{2}(U)=\frac{1}{\sigma}$. 

For such an $X$, it satisfies $11(m\uu+(n+1)\uu) \cdot kc^{2} \le 11(mn\uu+n(n+1)\uu) \cdot \norm{X}_{g}^{2}$ as shown in Theorem~\ref{THM 4.1} with $c=10$ and $k=2174$. Therefore, for the alternative choice of $s$, we have $s=j_{s}=11(m\uu+(n+1)\uu) \cdot kc^{2}$. We vary $\beta$ from $10^{-4}$, $10^{-6}$, $10^{-8}$, $10^{-10}$ to $10^{-12}$ to adjust $\kappa_{2}(X)$. We compare the numerical results with our alternative choice of $s$ and the original $s$ in \cite{New}. Meanwhile, for the dense matrix $U$ in the comparison group, $\sigma$ of $U$ varies to ensure $\kappa_{2}(U) \approx \kappa_{2}(X)$. For $U$, we use the original $s=11(mn\uu+n(n+1)\uu)\norm{X}_{g}^{2}$ in \cite{New}. We test the applicability and accuracy of SCholeskyQR3 with different $s$ for $X$ and $U$. We show the comparison of the applicability and accuracy in Table~\ref{tab:8}–Table~\ref{tab:10}. The accuracy is divided into two parts, the orthogonality $\norm{Q^{\top}Q-I}_{F}$ and residual $\norm{QR-X}_{F}$.

\begin{table}
\caption{Example I: SCholeskyQR3 for $X$ with the alternative choice of $s$}
\centering
\begin{tabular}{|c c c c c c|}
\hline
$\kappa_{2}(X)$ & $4.28e+06$ & $4.11e+08$ & $3.94e+10$ & $3.75e+12$ & $3.56e+14$ \\
\hline
Orthogonality & $2.29e-15$ & $3.33e-15$ & $1.53e-14$ & $1.92e-14$ & $3.14e-14$ \\
\hline
Residual & $8.32e-14$ & $8.24e-14$ & $8.28e-14$ & $8.15e-14$ & $8.21e-14$ \\
\hline
\end{tabular}
\label{tab:8}
\end{table}

\begin{table}
\caption{Example I: SCholeskyQR3 for $X$ with the original $s$}
\centering
\begin{tabular}{|c c c c c c|}
\hline
$\kappa_{2}(X)$ & $4.28e+06$ & $4.11e+08$ & $3.94e+10$ & $3.75e+12$ & $3.56e+14$ \\
\hline
Orthogonality & $2.68e-15$ & $1.33e-14$ & $2.15e-14$ & $3.33e-14$ & $-$ \\
\hline
Residual & $1.62e-13$ & $1.62e-13$ & $1.62e-13$ & $1.62e-13$ & $-$ \\
\hline
\end{tabular}
\label{tab:9}
\end{table}

\begin{table}
\caption{Example I: SCholeskyQR3 for $U$ with the original $s$}
\centering
\begin{tabular}{|c c c c c c|}
\hline
$\kappa_{2}(X)$ & $4.28e+06$ & $4.11e+08$ & $3.94e+10$ & $3.75e+12$ & $3.56e+14$ \\
\hline
Orthogonality & $1.71e-15$ & $2.06e-15$ & $2.03e-15$ & $2.12e-15$ & $-$ \\
\hline
Residual & $7.08e-16$ & $6.46e-16$ & $6.01e-16$ & $5.70e-16$ & $-$ \\
\hline
\end{tabular}
\label{tab:10}
\end{table}

According to Table~\ref{tab:8} and Table~\ref{tab:9}, we find that SCholeskyQR3 with our alternative choice of $s$ can handle more ill-conditioned sparse matrices than with the original $s$ in \cite{New}, demonstrating the improvement of our new $s$ for $T_{1}$ matrices in terms of applicability when $11(m\uu+(n+1)\uu) \cdot kc^{2} \le 11(mn\uu+n(n+1)\uu) \cdot \norm{X}_{g}^{2}$. When $\kappa_{2}(X) \ge 10^{14}$, our alternative choice of $s$ remains applicable, while the original $s$ does not. The comparison between Table~\ref{tab:8} and Table~\ref{tab:10} highlights the effectiveness of designing an alternative choice of $s$ for the sparse cases, which corresponds to the comparison in Table~\ref{tab:Comparisons1}. Theoretical results for $s=11(mn\uu+n(n+1)) \cdot \norm{X}_{g}^{2}$ in \cite{New} hold for all types of $X$. CholeskyQR-type algorithms perform differently between the common dense and the sparse cases. From the perspective of accuracy, for $X$ with the dense columns, SCholeskyQR3 with our alternative choice of $s$ performs better than the case with the original $s$ in both orthogonality and residual, as indicated by the comparison of orthogonality and residuals in Table~\ref{tab:8}–Table~\ref{tab:10}. This aligns with the theoretical results presented in Table~\ref{tab:Comparisons2}.

\subsubsection{Example II: the matrix with all sparse columns}
Sparse matrices with all the columns sparse are also very common in real applications, such as scientific computing, machine learning, and image processing \cite{Systems, Introduction}. In this group of numerical experiments, we take a $X \in \mathbb{R}^{2048\times 64}$. We define a vector $u_{t} \in \mathbb{R}^{64}$ as $u_{t}=(10,1,\cdots, 1, 1)^{\top}$. We define a diagonal matrix $E_{s}={\rm diag}(10, 10 \cdot b^{\frac{1}{63}}, \cdots, 10 \cdot b^{\frac{62}{63}}, 10 \cdot b) \in \mathbb{R}^{2048\times 64}$. Moreover, we build two matrices with all the elements $0$, $\mathbb{O}_{1984\times 64}$ and $\mathbb{O}_{1023\times 64}$. Therefore, a matrix $P_{sparse} \in \mathbb{R}^{2048\times 64}$ is formed as
\begin{equation}
D_{sparse}=
\begin{pmatrix}
E_{s} \\
\mathbb{O}_{1984\times 64} \nonumber
\end{pmatrix}.
\end{equation} 
Moreover, we construct $C_{sparse} \in \mathbb{R}^{2048\times 64}$ as
\begin{equation}
C_{sparse}=
\begin{pmatrix}
\mathbb{O}_{1023\times 64} \\
u_{t} \\
u_{t} \\
\mathbb{O}_{1023\times 64} \nonumber
\end{pmatrix}.
\end{equation} 
For another $K \in \mathbb{R}^{m\times n}$, each element $K_{ij}$ is defined as
$K_{ij}=
\begin{cases} 
10, & \text{if } j=1, i=2,3,\cdots,10 \\
0, & others
\end{cases}.$
We build $X \in \mathbb{R}^{2048\times 64}$ as
\begin{equation}
X=C_{sparse}+D_{sparse}+J. \nonumber
\end{equation}
For such an $X \in \mathbb{R}^{m\times n}$, we also take a dense $U \in \mathbb{R}^{2048\times 64}$ constructed by SVD as a comparison group.

For this $X$, $11(m\uu+(n+1)\uu) \cdot kc^{2} \le 11(mn\uu+n(n+1)\uu) \cdot \norm{X}_{g}^{2})$ still holds with $c=10$ and $k=201$. Therefore, we have $s=j_{s}=11(m\uu+(n+1)\uu) \cdot kc^{2}$ in \eqref{eq:js1} as shown in Theorem~\ref{THM 4.1}. We vary $b$ from $10^{-5}$, $10^{-7}$, $10^{-9}$, $10^{-11}$ to $10^{-13}$ to adjust $\kappa_{2}(X)$. We set the same comparison group as the previous experiments and take $s=11(mn\uu+n(n+1)\uu)\norm{X}_{g}^{2}$. We test the applicability of SCholeskyQR3 with different $s$ for both $X$ and $U$. The corresponding results of this group of numerical experiments are listed in Table~\ref{tab:8a}-Table~\ref{tab:10a}.

\begin{table}
\caption{Example II: SCholeskyQR3 for $X$ with the alternative choice of $s$}
\centering
\begin{tabular}{|c c c c c c|}
\hline
$\kappa_{2}(X)$ & $3.21e+06$ & $3.07e+08$ & $2.92e+10$ & $2.77e+12$ & $2.62e+14$ \\
\hline
Orthogonality & $1.89e-15$ & $2.16e-15$ & $1.91e-15$ & $1.86e-15$ & $1.77e-15$ \\
\hline
Residual & $4.18e-15$ & $3.28e-15$ & $3.51e-15$ & $6.14e-15$ & $4.55e-15$ \\
\hline
\end{tabular}
\label{tab:8a}
\end{table}

\begin{table}
\caption{Example II: SCholeskyQR3 for $X$ with the original $s$}
\centering
\begin{tabular}{|c c c c c c|}
\hline
$\kappa_{2}(X)$ & $3.21e+06$ & $3.07e+08$ & $2.92e+10$ & $2.77e+12$ & $2.62e+14$ \\
\hline
Orthogonality & $2.21e-15$ & $2.31e-15$ & $1.83e-15$ & $2.15e-15$ & $-$ \\
\hline
Residual & $7.93e-15$ & $7.29e-15$ & $7.24e-15$ & $8.13e-15$ & $-$ \\
\hline
\end{tabular}
\label{tab:9a}
\end{table}

\begin{table}
\caption{Example II: SCholeskyQR3 for $U$ with the original $s$}
\centering
\begin{tabular}{|c c c c c c|}
\hline
$\kappa_{2}(X)$ & $3.21e+06$ & $3.07e+08$ & $2.92e+10$ & $2.77e+12$ & $2.62e+14$ \\
\hline
Orthogonality & $1.76e-15$ & $2.15e-15$ & $2.10e-15$ & $2.09e-15$ & $-$ \\
\hline
Residual & $7.09e-16$ & $6.71e-16$ & $5.89e-16$ & $5.85e-16$ & $-$ \\
\hline
\end{tabular}
\label{tab:10a}
\end{table}

According to Table~\ref{tab:8a}-Table~\ref{tab:10a}, we observe that similar results hold for the sparse matrices without dense columns. Based on all the experiments regarding two examples, we can say that with our alternative choice of $s$, SCholeskyQR3 performs better in both applicability and accuracy compared to the case with the original $s$ in \cite{New}. We can provide different analysis and improvements for the sparse cases regarding CholeskyQR-type algorithms.

\subsection{CPU time(s) of the algorithm}
In this part, we focus on CPU time(s) of SCholeskyQR3 with our alternative choice of $s$ for sparse matrices. For the sparse $X \in \mathbb{R}^{m\times n}$, we focus on the influence of $m$, $n$ and the nnze and make comparison between SCholeskyQR3 with different $s$. 

To test the influence of the nnze on CPU time(s) of the algorithm, for the input $X \in \mathbb{R}^{m\times n}$, we fix $m=2000$ and $n=50$. We vary $h=\frac{nnze}{mn}$ from $0.02$, $0.05$, $0.10$, $0.15$ to $0.20$ and do numerical experiments. To test the influence of $m$, we fix $n=50$ and $h=0.05$. To test the influence of $n$, we fix $m=2000$ and $h=0.05$. The sparse $X \in \mathbb{R}^{m\times n}$ is built using $Sprand$ in MATLAB. We keep $\mbox{rank}(X)=n$ and $\kappa_{2}(X)=10^{6}$. We compare CPU time(s) of SCholeskyQR3 with our alternative choice of $s$ and the original $s$ in \cite{New}. Numerical results are shown below in Table~\ref{tab:13}-Table~\ref{tab:15}.

\begin{table}
\caption{Comparison of CPU time(s) with different $nnze$}
\centering
\begin{tabular}{|c c c c c c|}
\hline
$h$ & $0.02$ & $0.05$ & $0.10$ & $0.15$ & $0.20$ \\
\hline
Original & $0.002$ & $0.014$ & $0.032$ & $0.035$ & $0.036$ \\
\hline
Alternative & $0.002$ & $0.014$ & $0.031$ & $0.034$ & $0.035$ \\
\hline
\end{tabular}
\label{tab:13}
\end{table}

\begin{table}
\caption{Comparison of CPU time(s) with different $m$}
\centering
\begin{tabular}{|c c c c c c|}
\hline
$m$ & $200$ & $500$ & $1000$ & $2000$ & $5000$ \\
\hline
Original & $0.001$ & $0.002$ & $0.006$ & $0.014$ & $0.028$ \\
\hline
Alternative & $0.001$ & $0.002$ & $0.006$ & $0.014$ & $0.028$ \\
\hline
\end{tabular}
\label{tab:14}
\end{table}

\begin{table}
\caption{Comparison of CPU time(s) with different $n$}
\centering
\begin{tabular}{|c c c c c c|}
\hline
$n$ & $20$ & $50$ & $100$ & $200$ & $500$ \\
\hline
Original & $0.001$ & $0.014$ & $0.050$ & $0.222$ & $1.604$ \\
\hline
Alternative & $0.001$ & $0.014$ & $0.050$ & $0.229$ & $1.642$ \\
\hline
\end{tabular}
\label{tab:15}
\end{table}

According to Table~\ref{tab:13}-Table~\ref{tab:15}, we find that for the sparse $X \in \mathbb{R}^{m\times n}$, our alternative choice of $s$ for SCholeskyQR3 does not influence the efficiency of the algorithm compared to the case with the original $s$. Although our new choice of $s$ does not accelerate the algorithm for the sparse cases, SCholeskyQR3 still has better applicability with our alternative choice of $s$ than with the original one. Generally speaking, we can say that our alternative choice of $s$ improves the properties of SCholeskyQR3 and is a better option than the original $s$ shown in \cite{New} based on all the numerical experiments in this work.

\section{Conclusions}
\label{sec:conclusions}
This study focuses on SCholeskyQR for sparse matrices. We propose a new shifted item $s$ for SCholeskyQR3 based on the nnze and the element with the largest absolute value of the input $X$. Therefore, an optimal $s$ can be taken for SCholeskyQR3 for the sparse cases after the comparison between our new $s$ and the $s$ in \cite{New}. We prove that such an alternative choice of $s$ can guarantee the accuracy of SCholeskyQR3 and can improve the applicability of the algorithm for the sparse $X$ after detailed numerical analysis. Numerical experiments confirm the advantage of our alternative choice of $s$ for SCholeskyQR3 in both applicability and accuracy over the original $s$ in sparse cases. Moreover, SCholeskyQR3 remains as efficient with our alternative choice of $s$ as with the original $s$ in \cite{New}. 

\section*{Acknowledgments}
We acknowledge the help of Professor Zhonghua Qiao from the Hong Kong Polytechnic University with regard to this topic. We are also grateful for the idea provided by Professor Tiexiang Li from Southeast University, and for the discussions with her about sparse matrices. Additionally, we thank Professor Valeria Simoncini from University of Bologna and Professor Michael Kwok-Po Ng from Hong Kong Baptist University for their valuable suggestions. We also extend our gratitude to Dr. Yuji Nakatsukasa from Oxford University for the tips on CholeskyQR. 

\section*{Conflict of interest}
The authors declare that they have no conflict of interest.
 
\section*{Data availability}
The authors declare that all data supporting the findings of this study are available within this article.

\bibliography{references}

\end{document}